\documentclass[11pt, reqno]{article}

\usepackage[utf8]{inputenc}
\usepackage[T1]{fontenc}

\usepackage{setspace}
\usepackage[margin=1in]{geometry}
\usepackage{parskip}
\usepackage[sc]{mathpazo}
\usepackage[T1]{eulervm}
\usepackage[scaled]{helvet}

\usepackage{amsmath,amsthm, amssymb, amsfonts}
\usepackage{mathtools}
\usepackage{graphicx}
\usepackage[hidelinks]{hyperref}

\theoremstyle{plain}
\newtheorem{theorem}{Theorem}
\newtheorem{lemma}[theorem]{Lemma}
\newtheorem{proposition}[theorem]{Proposition}
\newtheorem{corollary}[theorem]{Corollary}

\theoremstyle{definition}

\newcommand{\ind}[1]{\mathbf{1}_{\{ #1 \}}}
\newcommand{\dt}[1]{\mathrm{det}\left(#1\right)}
\newcommand{\R}{\mathbb{R}}
\newcommand{\Z}{\mathbb{Z}}
\newcommand{\1}{\mathbf{1}}
\newcommand{\Jext}{\mathcal{K}^{\mathrm{ext}}}

\title{The transition between order to disorder in the Sepp\"al\"ainen-Johansson model}
\author{Mustazee Rahman\thanks{\textsc{Department of Mathematical Sciences, Durham University}. \textit{Email}: \texttt{mustazee@gmail.com}}}
\date{}

\begin{document}

\maketitle

\begin{abstract}
    \noindent We study the Sepp\"al\"ainen-Johansson model of first-passage percolation along a region of the plane where it moves
    from an ordered to a disordered phase, represented by a 2nd order discontinuity in its limit shape function.
    The fluctuations in this transition region are of order one, and we determine that certain limiting statistics are described by Fredholm determinants and are closely related to the largest eigenvalue of the GUE.
\end{abstract}

\section{Introduction} \label{sec:intro}

Sepp\"al\"ainen \cite{Sep} introduced a simplified model of planar first-passage percolation as follows.
Let $B_{i,j}$ for $(i,j) \in \Z^2$ be a family of independent, $\mathrm{Bernoulli(p)}$ random variables.
Direct every vertical edge of $\Z^2$ upward and every horizontal edge rightward. Assign weights to the
directed edges of $\Z^2$ as follows.
\begin{itemize}
    \item Assign to the vertical edge $(i,j) \to (i,j+1)$ a weight of zero.
    \item Assign to the horizontal edge $(i,j) \to (i+1,j)$ the random weight $B_{i,j}$.
\end{itemize}
A directed path $\pi :(a,b) \to (c,d)$ is a path that traverses the edges of $\Z^2$ according to their direction.
So such paths will move only up/right; e.g., $(0,0) \stackrel{up}{\to} (0,1) \stackrel{right}{\to} (1,1) \stackrel{up}{\to} (1,2)$ is a directed path.
The weight of a directed path $\pi : (a,b) \to (c,d)$ is the sum of all edge weights along the path:
$$ w_B(\pi) = \sum_{e \in \pi} w(e).$$
The first passage time from $(a,b)$ to $(c,d)$ is
\begin{equation} \label{eqn:Fptime}
    F_p(a,b;c,d) = \min_{\pi: (a,b) \to (c,d)}\, w_B(\pi).
\end{equation}

This model can be recast as a last-passage percolation model, as observed by Johansson \cite{Joplanc}.
Let $\bar{B}_{i,j} = 1 - B_{i,j}$ so that $\bar{B}_{i,j} \sim \mathrm{Bernoulli}(1-p)$.
Suppose all horizontal edges of $\Z^2$ now have corresponding weights $\bar{B}_{i,j}$ instead of $B_{i,j}$.
Let the weight $w_{\bar{B}}(\pi)$ of a directed path $\pi$ be the sum of edge weights along it for the new weights $\bar{B}_{i,j}$.
It is clear that if $\pi$ is a directed path from $(a,b) \to (c,d)$, then
$$ w_{\bar{B}}(\pi) = c-a - w_{B}(\pi).$$
The last passage time from $(a,b)$ to $(c,d)$ is
\begin{equation*} \label{eqn:Lptime}
    L_{1-p}(a,b;c,d) = \max_{\pi: (a,b) \to (c,d)} \, w_{\bar{B}}(\pi).
\end{equation*}
Owing to the relation between $w_B$ and $w_{\bar{B}}$ above, it is clear that
\begin{equation*} \label{eqn:FpLprelation}
    L_{1-p}(a,b;c,d) = c-a - F_p(a,b;c,d).
\end{equation*}
This model is thus called the Sepp\"al\"ainen-Johansson (SJ) model.

In this article, we study the SJ model at $p = \frac{1}{2}$.
We shall write $F(m,n)$ for $F_{1/2}(0,0;m,n)$.
By Kingman's subadditive ergodic theorem, there is a function $\mu$ (called the "limit shape" or "time constant") such that
$$ \frac{F(\lfloor \rho n \rfloor, n)}{n} \stackrel{a.s.}{\longrightarrow} \mu(\rho), \quad \rho \geq 0.$$
Sepp\"al\"ainen \cite{Sep} obtained the formula for the limit shape:
\begin{equation} \label{eqn:limitshape}
\mu(\rho) = \begin{cases}
    \frac{(\sqrt{\rho}-1)^2}{2}, & \rho \geq 1;\\
    0, & \rho \in [0,1]
\end{cases}.
\end{equation}
Limit shapes of generalisations of the SJ model are also known; see, for example, \cite{BC, Mar, OC, Ran}.

The function $\mu(\rho)$ is continuously differentiable but not twice differentiable at $\rho = 1$.
In this article, we shall investigate the behaviour of the first passage time $F(\rho n, n)$ at $\rho = 1$.
It is not hard to see that when $\rho < 1$, there exists a random integer $N$ such that $F(\rho n,n) = 0$ for all $n \geq N$.
We can therefore think of $\rho < 1$ as being the "ordered" phase of the SJ model.
In the "disordered" phase $\rho > 1$, Johansson \cite{Joplanc} proved that there exists a $\sigma_{\rho} > 0$ such that
\begin{equation} \label{eqn:JoTW}
\lim_n \mathbb{P}\left (\frac{\mu(\rho)n-F(\lfloor \rho n \rfloor,n)}{\sigma_{\rho} n^{1/3}} \leq \xi\right ) = F_2(\xi),
\end{equation}
where $F_2$ is the distribution function of the GUE Tracy-Widom law \cite{TW}.
The quantity $\sigma_{\rho} \sim (\sqrt{\rho}-1)^{2/3}$ as $\rho \to 1$. Therefore, the behaviour of the model
should change drastically at $\rho = 1$, which motivates our investigation. When $\rho > 1$,
Dauvergne, Nica and Vir\'ag \cite{DNV} have shown that there are constants $c_{\rho}, c'_{\rho} >0$ such that the process
\begin{equation} \label{eqn:DNV}
t \mapsto \frac{\mu(\rho)n+ c'_{\rho} n^{2/3}t-F(\lfloor \rho n + c_{\rho}n^{2/3}t \rfloor,n)}{\sigma_{\rho} n^{1/3}} \to \mathcal{A}(t)
\end{equation}
in law, where $\mathcal{A}(t)$ is the (parabolic) Airy-two process \cite{JoDPG, PS}.

The SJ model is also related to the Krawtchouk orthogonal polynomial ensemble and, through this, to certain statistics in a uniform random tiling of the Aztec Diamond.
See \cite{Joplanc} for more on this relationship. It is also related to the discrete-time totally asymmetric simple exclusion process with sequential updates; see \cite{DW} Case B.

When $\rho = 1$, $F(n,n)$ fluctuates at order one. In fact, it converges
to an integer-valued random variable whose c.d.f.~is related to the largest eigenvalue from the Gaussian Unitary Ensemble.
It also coincides with a certain point-to-set distance in the Brownian web distance model introduced by Vet\H{o} and Vir\'ag \cite{VV}.
Our results provide Fredholm determinant formulas for the c.d.f.s of such distances from their model; see also the forthcoming article \cite{RVV}.

\section{Results} \label{sec:results}
Define the following integral kernel $J$ acting on $\ell^2(\Z)$:
\begin{equation} \label{eqn:J}
    J(x,y) = \frac{1}{(2 \pi \mathbf{i})^2} \oint_{d+\mathbf{i}\R} dz_1 \oint_{|z_0|=1} dz_0\, e^{\frac{1}{2}(z_1^2-z_0^2)} \, \frac{z_1^{y}}{z_0^{x+1}} \cdot \frac{1}{z_1-z_0}.
\end{equation}
Here $d > 1$; the vertical line $d + \mathbf{i}\R$ is oriented upwards and the circle $|z_1|=1$ is oriented counter-clockwise.

\begin{theorem} \label{thm:main}
    The random variable $F(n,n) = F_{1/2}(0,0;n,n)$ from \eqref{eqn:Fptime} satisfies, for every $a \in \Z$,
    $$ \lim_n\, \mathbb{P}(F(n,n) \geq a) = \det\left (I- \chi_{[0,a)} J \chi_{[0,a)}\right )_{\ell^2(\Z)}.$$
    Here $\chi_{[0,a)}(x) = \ind{x \in [0,a)}$ is the projection onto the set $[0,a) = \{0,\cdots,a-1\}$.
    In particular, $\chi_{[0,a)} J \chi_{[0,a)}$ is non-zero only if $a > 0$ and the Fredholm determinant is a finite $a \times a$ matrix determinant.
    Furthermore, $a \in \Z \mapsto 1- \det(I-\chi_{[0,a)} J \chi_{[0,a)})_{\ell^2(\Z)}$ is a cumulative distribution function.
\end{theorem}
In fact, we will demonstrate:
\begin{corollary} \label{cor:main}
    For every integer $a \geq 1$,
    $$\det\left (I- \chi_{[0,a)} J \chi_{[0,a)}\right )_{\ell^2(\Z)} = \mathbb{P}\left( \lambda_{max}(\mathbf{G}_{a\times a})\leq 0 \right),$$
    where $\lambda_{max}(\mathbf{G}_{a \times a})$ is the largest eigenvalue of an $a \times a$ matrix from the Gaussian Unitary Ensemble.
    Specifically, $\mathbf{G}_{a \times a} = \frac{W + W^*}{\sqrt{2}}$ where the $a \times a$ matrix $W$ has independent entries $W_{i,j} \sim \mathbf{N}_{\mathbb{C}}(0,1)$.
\end{corollary}

We will also see that
\begin{equation} \label{eqn:JKrelation}
\det\left (I- \chi_{[0,a)} J \chi_{[0,a)}\right )_{\ell^2(\Z)} = \det \left( I - K_a\right)_{L^2([0,\infty))},
\end{equation}
where the integral kernel $K_a$ is
$$ K_a(u,v) = \frac{1}{(2 \pi \mathbf{i})^2} \oint_{d+\mathbf{i}\R} dz_1 \oint_{|z_0|=1} dz_0\,
e^{\frac{1}{2}(z_1^2-z_0^2)+uz_0-vz_1} \Big ( \frac{z_1}{z_0}\Big)^a\, \frac{1}{z_1-z_0}.$$
It is possible to express $K_a$ in terms of Hermite polynomials:
\begin{equation}
    K_a(u,v) = \frac{\ind{a \geq 1}}{\sqrt{2\pi}(a-1)!}\int_0^{\infty} d\lambda\, H_a(v+\lambda) H_{a-1}(u+\lambda) e^{-\frac{(v+\lambda)^2}{2}},
\end{equation}
where $H_n(x) = e^{x^2/2} \cdot  \big (- \frac{d}{dx}\big)^n e^{-x^2/2}$ are the Hermite polynomials.

At the process level, $F(n+t,n)$ is non-trivially correlated to $F(n,n)$ when $t \sim \sqrt{n}$. This contrasts with the case $\rho > 1$, as exhibited in \eqref{eqn:DNV}, where non-trivial correlations arise when $t \sim n^{2/3}$.
\begin{theorem} \label{thm:process}
    For $t \in \R$, define the process $\mathcal F_n(t) = F(n+\lfloor \sqrt{2n}\, t\rfloor,n)$. Given times $t_1 < t_2 < \cdots < t_k$ and integers $a_1, \ldots, a_k$,
    $$\lim_n\, \mathbb{P}(\mathcal F_n(t_i) \geq a_i, 1 \leq i \leq k) = \det(I - \chi_{[0,a)} \mathcal{K} \chi_{[0,a)})_{\ell^2(\Z \times \{t_1, \ldots, t_k\})},$$
    where $\chi_{[0,a)}(x,t_i) = \ind{0 \leq x < a_i}$ and the kernel
    \begin{align*}
        \mathcal{K}(x,s;y,t) & = - \ind{s>t, x\geq y} \frac{(s-t)^{x-y}}{(x-y)!} + \\
        & \frac{1}{(2 \pi \mathbf{i})^2} \oint_{d+\mathbf{i}\R} dz_1 \oint_{|z_0|=1} dz_0\, \frac{e^{\frac{1}{2}z_1^2-tz_1}}{e^{\frac{1}{2}z_0^2-sz_0}} \, \frac{z_1^y}{z_0^{x+1}} \cdot \frac{1}{z_1-z_0}.
    \end{align*}
    Here, $d> 1$, the vertical line $d+\mathbf{i}\R$ is oriented upwards and the circle $|z_1|=1$ is oriented counter-clockwise.
\end{theorem}
One can show by arguing as in the proof of Theorem \ref{thm:main} that the Fredholm determinants in Theorem \ref{thm:process} determine a consistent family of probability distributions and, as a result, determine an integer-valued limit process
$$\mathcal F(t) \stackrel{d}{=} \lim_n\, \mathcal F_n(t).$$
The process $\mathcal F(t)$ is non-decreasing as $m \mapsto F(m,n)$ is non-decreasing.
We will establish:
\begin{corollary} \label{cor:extendedkernel}
    The determinant
    $$\det(I - \chi_{[0,a)} \mathcal{K} \chi_{[0,a)})_{\ell^2(\Z \times \{t_1, \ldots, t_k\})} = \det(I - \chi \mathcal{K}^{ext} \chi)_{\ell^2(\Z \times \{t_1, \ldots, t_k\})},$$
    where $\chi(x,t_i) =\chi_{[0,a)}(x,t_i) = \ind{0 \leq x < a_i}$ and
    $$\mathcal{K}^{ext}(x,s;y,t) = \begin{cases}
        \displaystyle \int_0^{\infty}d\lambda\,  \frac{H_{x}(\lambda+s) H_{y}(\lambda+t)}{y!} \frac{e^{-\frac{(\lambda+s)^2}{2}}}{\sqrt{2\pi}} & t \leq s \\[2ex]
        
        \displaystyle -\int_{-\infty}^{0}d\lambda\,  \frac{H_{x}(\lambda+s) H_{y}(\lambda+t)}{y!} \frac{e^{-\frac{(\lambda+s)^2}{2}}}{\sqrt{2\pi}} & t > s
    \end{cases}.$$
\end{corollary}
Here $H_0, H_1, H_2, \ldots$ are the family of Hermite polynomials as defined above. (See \cite{Suetin} for a reference on Hermite polynomials.)

Arguing as in the proof of Corollary \ref{cor:main}, one can show that
\begin{equation} \label{eqn:GUErelation}
\mathbb{P}(\mathcal F(t) \geq a) = \mathbb{P}(\lambda_{max}(\mathbf{G}_{a \times a}) \leq t),\quad t \in \R, a \in \{1,2,3,\ldots\}.
\end{equation}
One could think of $\mathcal{F}(t)$ as the analogue of the Airy process $\mathcal{A}(t)$ from \eqref{eqn:DNV} in this setting.

A more general version of the process $\mathcal{F}(t)$ is constructed in the forthcoming article \cite{RVV}.
The authors construct the large $n$ scaling limit of the process
$$(x,s;y,t) \mapsto F_{1/2}((ns-\sqrt{n}x)/2,(ns+\sqrt{n}x)/2;(nt-\sqrt{n}y)/2,(nt+\sqrt{n}y)/2),$$
and call it the symmetric Brownian web distance $D^{sBr}(x,s;y,t)$. This is closely related to the Brownian web distance model \cite{VV}.
It is shown in \cite{RVV} that as a process in $t$,
\begin{equation} \label{eqn:RVV} \mathcal F(t) = D^{sBr}(0,0; t,1) = \inf_{s \geq t}\, D^{Br}(0,0;s,1),\end{equation}
where $D^{Br}$ is the Brownian web distance from \cite{VV}.

Johansson's result \eqref{eqn:JoTW} can be recovered by considering the large $t$ limit of $\mathcal F$:
$$ \lim_{t \to \infty}\, \mathbb{P}\left( \frac{(t/2)^2 - \mathcal F(t)}{(t/2)^{2/3}} \leq \xi \right) = F_2(\xi).$$
This may be proved by analysing the Fredholm determinant from Theorem \ref{thm:process} or by \eqref{eqn:GUErelation}.
As $t \to -\infty$, $\mathcal F(t) \stackrel{a.s.}{\to} 0$ as expected, which follows, for example, from \eqref{eqn:RVV}.

We conclude the Introduction with the following natural extension of Corollary \ref{cor:main}\footnote{After discussion with the authors of \cite{RVV} following the completion of this manuscript, we realised that the theorem also follows from their work, albeit with a very different proof, which is interesting.}:
\begin{theorem} \label{thm:extendedcorollary}
    Let $W$ be a $\{1,2,3,\ldots\} \times \{1,2,3,\ldots\}$ random matrix consisting of independent entries $W_{i,j} \sim \mathbf{N}_{\mathbb{C}}(0,1)$
    (standard, complex-valued Normal random variables). Let $\mathbf{G} = \frac{W + W^*}{\sqrt{2}}$ and denote by $\mathbf{G}_{a \times a}$ the top-left
    $a \times a$ submatrix of $\mathbf{G}$. Then for $t_1, \ldots, t_k \in \R$ and $a_1, \ldots, a_k \in \{1,2,3,\ldots\}$,
    $$\mathbb{P}(\mathcal{F}(t_i) \geq a_i, 1 \leq i \leq k) = \mathbb{P}(\lambda_{max}(\mathbf{G}_{a_i \times a_i}) \leq t_i, 1 \leq i \leq k).$$
\end{theorem}
A way to interpret this theorem is that $\mathcal{F}(t)$ starts from $\mathcal{F}(-\infty)=0$ and jumps 1 unit at the times $\lambda_{max}(\mathbf{G}_{1\times 1}) < \lambda_{max}(\mathbf{G}_{2 \times 2}) < \cdots.$

\section{Preliminaries} \label{sec:prel}

Let $K$ be an integral kernel acting on the space $L^2(\Omega,\mu)$. The Fredholm determinant of $K$ is
\begin{equation} \label{eqn:Fredholm}
\dt{I-K}_{L^2(\Omega,\mu)} = 1 + \sum_{k=1}^{\infty} \frac{(-1)^k}{k!} \int_{\Omega^K} d\mu(z_1) \cdots d\mu(z_k)\, \dt{K(z_i,z_j)}_{1\leq i,j\leq k.}\end{equation}
If there are functions $f$ and $g$ on $\Omega$ such that
$$ |K(x,y)| \leq f(x) g(y)$$
with $f$ bounded and $g$ integrable (or vice versa), then the series converges absolutely.
Furthermore, suppose a sequence of integral kernels
$K_n$ satisfy $K_n \to K$ pointwise on $\Omega$ and $|K_n(x,y)| \leq f(x) g(y)$ for every $n$ with $f$ bounded and $g$ integrable (or vice versa). Then,
$$ \dt{I-K_n}_{L^2(\Omega,\mu)} \to \dt{I-K}_{L^2(\Omega,\mu)}.$$
See \cite{Josurvey} for proofs of these facts, which are deduced from Hadamard's inequality and the dominated convergence theorem.

A conjugation of an integral kernel $K$ is a kernel $K'$ of the form
$$ K'(x,y) = \frac{c(x)}{c(y)} K(x,y)$$
where $c: \Omega \to \R\setminus{\{0\}}$ is a non-vanishing function.
Fredholm determinants remain invariant under conjugation: $\dt{I-K'} = \dt{I-K}$.
However, bounds of the form $|K(x,y)| \leq f(x)g(y)$ are \emph{not} invariant under conjugation.

To demonstrate that the series expansion of a Fredholm determinant is absolutely convergent, it is often necessary to conjugate the kernel $K$ so that the conjugated kernel $K'$ does satisfy a bound of the form $|K'(x,y)| \leq f(x)g(y)$ with $f$ bounded and $g$ integrable (or vice versa).
For simplicity, it is customary not to include such conjugation factors when presenting a kernel in a theorem. However, in proofs, we do conjugate kernels so that the Fredholm series expansion converges absolutely.

The proofs of Theorem \ref{thm:main} and Theorem \ref{thm:process} are based on a Fredholm determinant formula for the law of the process $t \mapsto F(t,n)$ as follows.
Let $X_{n,1}(t) < X_{n,2}(t) < \cdots < X_{n,n}(t)$ be $n$ non-colliding Bernoulli random walks started from $X_{n,i}(0) = i-1$.
Here, non-colliding means that we condition $n$ independent Bernoulli random walks (with step distribution $\mathrm{Ber}(1/2)$) to never collide via the Doob $h$-transform where $h$ is the Vandermonde; see \cite{KOCR}. Then it is shown in \cite[Theorem 4.6]{KOCR} that the bottom-most walk $t \mapsto X_{n,1}(t)$ has the same law as the process $t \mapsto F(t,n)$.
A formula for the finite-dimensional distributions of $X_{n,1}$ is given in \cite[Display (11)]{DNV}, by way of \cite[Theorem 2.25]{BF} (see also \cite[Proposition 5.5]{BG}).
We quote the formula here.

\begin{proposition} \label{prop:BF}
    Consider the first passage function $F(t,n) = F_{1/2}(0,0;t,n)$ from \eqref{eqn:Lptime}.
    For $1 \leq t_1 < t_2 < \cdots < t_k$ and $n \geq 1$, we have
   $$ \mathbb{P}(F(t_i,n) \geq a_i ) = \det(I- \chi_a K \chi_a)_{\ell^2(\Z \times \{t_1, \ldots, t_k\})}.$$
    Here $\chi_a(x,t_i) = \ind{x < a_i}$ and
    $$K(x,s;y,t) = - H_{s-t}(x,y)\, \ind{s>t} + J(x,s;y,t).$$
    The kernel $H_k$ equals
    $$ H_k(x,y) = \binom{k}{x-y} = \frac{k!}{(x-y)!(k-x+y)!}\ind{0 \leq x-y \leq k}.$$
    The kernel
    $$J(x,s;y,t) =\frac{1}{(2\pi \mathbf{i})^{2}} \oint_{\gamma_0}d\zeta_0 \oint_{\gamma_{1}}d\zeta_1\,
    \frac{(1+\zeta_0)^s(1-\zeta_0)^n\zeta_0^{-x-1}}{(1+\zeta_1)^t(1-\zeta_1)^n\zeta_0^{-y}}\, \frac{1}{\zeta_0-\zeta_1}.$$
    The contour $\gamma_1$ encloses $1$ and $\gamma_{0}$ encloses $0$; they are oriented counter-clockwise and do not intersect.
\end{proposition} 

\section{Proof of Theorem \ref{thm:main}} \label{sec:mainproof}
According to Proposition \ref{prop:BF},
\begin{equation} \label{eqn:mainproof1}
    \mathbb{P}(F(n,n) \geq a) = \det \left( I - \chi_{a} J_n \chi_{a}\right)_{\ell^2(\Z)},
\end{equation}
where $\chi_{a}(x) = \ind{x < a}$ and the kernel $J_n(x,y)$ has the formula
$$J_n(x,y) = \frac{1}{(2\pi \mathbf{i})^{2}} \oint_{\gamma_0}d\zeta_0 \oint_{\gamma_{1}}d\zeta_1\,
    \left(\frac{1-\zeta_0^2}{1-\zeta_1^2}\right)^n \frac{\zeta_1^y}{\zeta_0^{x+1}} \cdot \frac{1}{\zeta_0-\zeta_1}.$$

Note that when $x \leq -1$, $J_n(x,y) =0$ as we can contract the $\zeta_0$-contour to a point. Let $\bar{\chi}_0(x) = \ind{x \geq 0}$.
Thus, $J_n = \bar{\chi}_0 J_n$. Since $\bar{\chi}_0^2 = \bar{\chi}_0$ and $\chi_a \bar{\chi}_0 = \bar{\chi}_0 \chi_a$,
by the cyclic property of the determinant (also known as the Sylvester identity):
$$\det(I - \chi_a J_n \chi_a)_{\ell^2(\Z)} = \det(I - \chi_a \bar{\chi}_0 J_n \chi_a \bar{\chi}_0)_{\ell^2(\Z)} = \det(I - \chi_{[0,a)}J_n \chi_{[0,a)})_{\ell^2(\Z)}.$$
In particular, the r.h.s.~is a finite $a \times a$ determinant when $a > 0$ and equals 1 when $a \leq 0$.
To derive the limit, we thus only need to take the entry-wise limit of the $a \times a$ kernel.

Next, we make explicit choices for the two contours $\gamma_0$ and $\gamma_1$. We set
\begin{align} \label{eqn:gamma0}
    \gamma_0 &= \gamma_0(\sigma_0) = (1/\sqrt{2n}) \cdot e^{\mathbf{i}\sigma_0}, \quad |\sigma_0| \leq \pi.\\ \label{eqn:gamma2}
    \gamma_1 &= \gamma_1(\sigma_1,d) = 1 - 1\Big( 1 - \frac{d}{\sqrt{2n}}\Big)e^{\mathbf{i}\sigma_1/\sqrt{2n}}, \quad |\sigma_1| \leq \pi \sqrt{2n}.
\end{align}
Thus, $\gamma_0$ is a circle around $0$ of radius $1/\sqrt{2n}$ and $\gamma_1$ is a circle centred at 1 of radius $1- d/\sqrt{2n}$.
We assume here that $d > 1$ is fixed.

We make the change of variables $\zeta_i = z_i/\sqrt{2n}$, which gives
$$J_n(x,y) = (\sqrt{2n})^{x-y} \frac{1}{(2\pi \mathbf{i})^{2}} \oint_{\gamma}dz_1 \oint_{|z_0|=1}dz_1\,
    \frac{(1-\frac{z_0^2}{2n})^n}{(1-\frac{z_1^2}{2n})^n}\, \frac{z_1^y}{z_0^{x+1}}\, \frac{1}{z_0-z_1}.$$
The contour $\gamma = \sqrt{2n} \times \gamma_1$. We can remove the conjugation factor $(\sqrt{2n})^{x-y}$ from the kernel without affecting the determinant.

Since the contour $|z_0| = 1$ is free of $n$, we see that $(1-\frac{z_0^2}{2n})^n \to e^{-z_0^2/2}$ as $n \to \infty$, uniformly over the integration contour.
On the other hand, along the contour $\gamma$, if $|z_1| \leq L$ then we find from Taylor's theorem that $z_1 = d - \mathbf{i}\sigma_0 + O_L(n^{-1/2})$ for $|\sigma_1| \leq L$.
Thus, locally, $\gamma$ is the vertical line $d + \mathbf{i}\R$ oriented downwards. When $|z_1| \leq L$, we have that $(1- z_1^2/2n)^n \to e^{-z_1^2/2}$.

In order to get the limit of $J_n$ we have to bound the global behaviour of $(1 - z_1^2/2n)^{-n}$ along the contour $\gamma$, or equivalently, the global behaviour of $(1-\zeta_1^2)^{-n}$ along the contour $\gamma_1$ from \eqref{eqn:gamma2}. We establish the following lemma.

\begin{lemma} \label{lem:mainglobal}
    Suppose $\zeta_1 \in \gamma_1(\sigma_1,d)$ from \eqref{eqn:gamma2}. Then, for all $n \geq n_0 = n_0(d)$ and $|\sigma_1| \leq \pi \sqrt{2n}$,
    $$\left | (1- \zeta_1^2)^{-n}\right| \leq e^{- \frac{1}{5 \pi^2} \sigma_1^2 + d^2}.$$
\end{lemma}

\begin{proof}
    Let $f(\zeta) = \log(1-\zeta^2)$ and $g(\sigma) = \Re(f(\zeta(\sigma)))$ for $\zeta = \zeta(\sigma) \in \gamma_1(\sigma,d)$.
    It suffices to show that $g(\sigma) \geq \sigma^2/(5 \pi^2 n) - d^2/n$ for $n \geq n_0(d)$.

    Since $g(\sigma) = g(-\sigma)$, we may assume that $0 \leq \sigma \leq \pi \sqrt{2n}$. We find that
    $$g(\sigma) = \log \left| 1 + \zeta\right| + \log \left | 1 - \zeta \right| = \log \left(2 \left | 1 - \alpha e^{\mathbf{i}\beta \sigma}\right|\right) + \log\left (1-\frac{d}{\sqrt{2n}}\right),$$
    where $\alpha = \frac{1}{2}(1 - d/\sqrt{2n})$ and $\beta = 1/\sqrt{2n}$. As a result, the derivative
    $$ g'(\sigma) = \frac{\alpha \beta \sin(\beta \sigma) }{(1-\alpha)^2 + 4\alpha \sin^2(\beta \sigma/2)}.$$
    
    Note that $0 < (1-\alpha)^2 + 4\alpha \sin^2(\beta \sigma/2) \leq (1-\alpha)^2 + 4\alpha = (1+\alpha)^2$.
    This gives
    \begin{equation} \label{eqn:gderivative}
    g'(\sigma) \geq \frac{ \alpha \beta}{(1+\alpha)^2}\, \sin(\beta \sigma), \quad 0 \leq \sigma \leq \pi \sqrt{2n}.
    \end{equation}
    Integrating \eqref{eqn:gderivative} and using the cosine double-angle formula gives
    $$ g(\sigma) - g(0) \geq \frac{2\alpha}{(1+\alpha)^2}\cdot \sin^2(\beta \sigma/2) \geq \frac{2\alpha}{(1+\alpha)^2} \left( \frac{\beta \sigma}{\pi}\right)^2,$$
    since $\sin(t) \geq (2/\pi)t$ for $0 \leq t \leq \pi/2$.
    
    If $n \geq n_0(d)$, then $\frac{\alpha}{(1+\alpha)^2} \geq \frac{1}{9}$. As a result, the bound above implies that
    $$ g(\sigma) - g(0) \geq \frac{\sigma^2}{5 \pi^2 n}.$$
    Finally, $g(0) = \log(1-d^2/2n) \geq -d^2/n$ for $n \geq n_0(d)$. This proves the required bound.
\end{proof}

Due to Lemma \ref{lem:mainglobal} and the dominated convergence theorem, in the limit, all contributions to $J_n(x,y)$ will come from the local contribution.
Along the chosen contours, we have that $|z_0|^{x+1} = 1$ and $|z_1| \leq 2$ so that $|z_1|^y \leq 2^{(a)_{+}}$ for all $0 \leq y < a$.
Also $|z_1-z_0| \geq d-1 > 0$. By Lemma \ref{lem:mainglobal},
$$\left | \Big ( \frac{1- z_0^2/2n}{1-z_1^2/2n}\Big)^n \right| \leq C e^{- \kappa \sigma_1^2}$$
for $(z_1,z_0) \in \gamma \times \{|z_0|=1\}$ and some positive constants $C, \kappa$.
Therefore,
$$ \left|\text{integrand of}\; J_n(x,y)\right| \leq C \, \frac{2^{(a)_{+}}}{d-1}\, e^{-\kappa \sigma_1^2}$$
for $(z_1,z_0) \in \gamma \times \{|z_0|=1\}$ and all $0 \leq x,y < a$.
If we orient the contour $\gamma$ clockwise, we pick up a minus sign, and we get by the dominated convergence theorem that
$$J_n(x,y) \to J(x,y)$$
as $n \to \infty$ for every $(x,y) \in [0,a)^2$.
As our Fredholm determinant is an $a \times a$ matrix determinant, we deduce that
$$\det(I - \chi_{a} J_n \chi_{a})_{\ell^2(\Z)} \to \det(I - \chi_{[0,a)}J\chi_{[0,a)})_{\ell^2(\Z)}.$$

Finally, we show that $a \mapsto 1 - \det(I - \chi_{[0,a)}J\chi_{[0,a)})_{\ell^2(\Z)}$ is a valid distribution function.
Note that $a \mapsto \det(I - \chi_{[0, a)}J\chi_{[0, a)})_{\ell^2(\Z)}$ is non-increasing because it is a limit of the non-increasing functions $\mathbb{P}(F(n,n) \geq a)$.
Also, $\det(I - \chi_{[0, a)}J\chi_{[0, a)})_{\ell^2(\Z)} = 1$ for every $a \leq 0$.
So we only need to show that $\det(I - \chi_{[0, a)}J\chi_{[0, a)})_{\ell^2(\Z)} \to 0$ as $a \to \infty$. To do so,
we will first establish Corollary \ref{cor:main} and deduce this as a consequence.

For integers $x,y \geq 0$ and a real number $\lambda \geq 0$, define the kernels
\begin{align*}
    A(x,\lambda) &=  \frac{1}{2\pi \mathbf{i}} \oint_{|z|=1r}dz\, e^{-\frac{1}{2}z^2+\lambda z} z^{-(x+1)}, \\
    B(\lambda,y) &=  \frac{1}{2\pi \mathbf{i}} \oint_{d+\mathbf{i}\R}dz\, e^{\frac{1}{2}z^2-\lambda z} z^{y}.
\end{align*}
Note that $\Re(z_1-z_0) > 0$ in the integrand of $J(x,y)$.
By writing $(z_1-z_0)^{-1} = \int_0^{\infty}d\lambda\, e^{-\lambda(z_1-z_0)}$, we see that $J(x,y) = AB(x,y)$ for every $x,y \geq 0$.

Note that $$ B(\lambda,y) = \left ( - \frac{d}{d\lambda}\right)^{y} \left [ \frac{1}{2\pi \mathbf{i}} \oint_{\mathbf{i}\R}dz\, e^{\frac{1}{2}z^2-\lambda z} \right].$$
We can evaluate the Gaussian integral, which equals
$$\frac{1}{2\pi \mathbf{i}} \oint_{\mathbf{i}\R}dz\, e^{\frac{1}{2}z^2-\lambda z} = \frac{e^{-\frac{\lambda^2}{2}}}{\sqrt{2 \pi}}.$$
As such, $$B(\lambda,y) = H_{y}(\lambda) \times \frac{e^{-\frac{\lambda^2}{2}}}{\sqrt{2 \pi}}$$
where $H_0,H_1,H_2,\ldots$ are the family of Hermite polynomials.

The generating function of the Hermite polynomials is
$$ e^{-\frac{z^2}{2} + \lambda z} = \sum_{n=0}^{\infty} H_n(\lambda) \frac{z^n}{n!}.$$
Consequently,
$$A(x,\lambda) = \frac{H_{x}(\lambda)}{x!}.$$

So for $x,y \geq 0$,
$$J(x,y) = \int_0^{\infty}d\lambda \, H_{y}(\lambda) \frac{H_{x}(\lambda)}{x!} \frac{e^{-\lambda^2/2}}{\sqrt{2\pi}}.$$
The Hermite polynomials are orthogonal for the density $w(\lambda) = e^{-\lambda^2/2}/\sqrt{2\pi}$ and satisfies \cite{Suetin}
$$\int_{-\infty}^{\infty}d\lambda \, H_{y}(\lambda) \frac{H_{x}(\lambda)}{x!} \frac{e^{-\lambda^2/2}}{\sqrt{2\pi}} = \ind{x=y} = I(x,y).$$
Therefore, for $x,y \geq 0$,
\begin{equation} \label{eqn:I-J}
    (I-J)(x,y) = \int_{-\infty}^0d\lambda\, H_{y}(\lambda) \frac{H_{x}(\lambda)}{x!} \frac{e^{-\lambda^2/2}}{\sqrt{2\pi}} =: M(x,y).
\end{equation}
Equation \eqref{eqn:I-J} implies that for every integer $a \geq 1$,
$$\det(I-\chi_{[0,a)} J \chi_{[0,a)})_{\ell^2(\Z)} = \det(M(x,y))_{0 \leq x,y \leq a-1}.$$

By the Cauchy-Binet identity,
\begin{align*}
    \det(M(x,y))_{0 \leq x,y \leq a-1} &= \frac{1}{a!} \int_{(-\infty,0]^a}d\lambda_0 \cdots d\lambda_{a-1}\\
    & \det(H_i(\lambda_j))_{0 \leq i,j \leq a-1} \det\left (\frac{H_i(\lambda_j)}{i!} \frac{e^{-\lambda_j^2/2}}{\sqrt{2\pi}}\right )_{0 \leq i,j \leq a-1}\\
    &= \frac{(2\pi)^{-a/2}}{\prod_{i=0}^a i!} \int_{(-\infty,0]^a}d\lambda_0 \cdots d\lambda_{a-1} \, \det(H_i(\lambda_j))^2 e^{-\frac{1}{2} \sum_{j=0}^{a-1}\lambda_j^2}\\
    & = \frac{(2\pi)^{-a/2}}{\prod_{i=1}^a i!} \int_{(-\infty,0]^a}d\lambda_0 \cdots d\lambda_{a-1} \,
    \Delta(\lambda_0,\ldots,\lambda_{a-1})^2 e^{-\frac{1}{2} \sum_{j=0}^{a-1}\lambda_j^2}.
\end{align*}
In the last line, we used the fact that $H_0, H_1, \ldots, H_{a-1}$ are monic polynomials with corresponding degrees $0, 1,\ldots, a-1$ to deduce that
$\det(H_i(\lambda_j))$ equals the Vandermonde determinant $$\Delta(\lambda_0,\ldots,\lambda_{a-1}) = \prod_{0\leq i<j\leq a-1} (\lambda_j-\lambda_i).$$

We recognise the function
$$(\lambda_0, \ldots, \lambda_{a-1}) \mapsto \frac{(2\pi)^{-a/2}}{\prod_{i=1}^a i!} \Delta(\lambda_0,\ldots,\lambda_{a-1})^2 e^{-\frac{1}{2} \sum_{j=0}^{a-1}\lambda_j^2}$$
as the joint density of the eigenvalues of the GUE random matrix $\mathbf{G}_{a \times a}$ \cite{For, Wg}. As a result,
$$ \det(I-\chi_{(-a,0]} J \chi_{(-a,0]})_{\ell^2(\Z)} = \det(M(x,y))_{0 \leq x,y \leq a-1} = \mathbb{P}(\lambda_{max}(\mathbf{G}_{a \times a})\leq 0),$$
which proves the corollary.

Wigner's semi-circle law for GUE \cite{For, Wg} dictates that $\lambda_{max}(\mathbf{G}_{a \times a})/ 2 \sqrt{a} \stackrel{p}{\to} 1$ as $a \to \infty$,
from which one readily deduces that $\det(I - \chi_{[0,a)}J\chi_{[0,a)})_{\ell^2(\Z)} \to 0$ as $a \to \infty$.
This completes the proof of Theorem \ref{thm:main}.

Finally, since $J = A B$, by the cyclic property of the Fredholm determinant:
$$ \det(I - \chi_{[0,a)} J \chi_{[0,a)})_{\ell^2(\Z)} = \det(I - B \chi_{[0,a)} A)_{L^2([0,\infty)])} = \det(I - K_a)_{L^2[0,\infty)},$$
where the kernel
\begin{align*}
K_a(u,v) & = \sum_{0 \leq x < a} B(u,x)A(x,v) \\
& = \frac{1}{(2\pi \mathbf{i})^2}\oint_{d+\mathbf{i}\R}dz_1 \oint_{|z_0|=1}dz_0\, e^{\frac{1}{2}(z_1^2-z_0^2)+uz_0-vz_1} \big(\frac{z_1}{z_0}\big)^a \frac{1}{z_1-z_0}.
\end{align*}

\section{Proof of Theorem \ref{thm:process}}
According to Proposition \ref{prop:BF}, for integers $-n<m_1<\cdots < m_k$,
$$\mathbb{P}(F(n+m_i,n) \geq a_i, 1 \leq i \leq k) = \det(I - \chi_a K_n \chi_a)_{\ell^2(\Z \times \{m_1, \ldots, m_k\})}$$
where $\chi_a(x,m_i) = \ind{x < a_i}$ and $K_n(x,s;y,t) = - H_{s-t}(x,y) \, \ind{s>t} + J_n(x,s;y,t)$ with
\begin{align*}
    J_n(x,s;y,t) &= \frac{1}{(2\pi \mathbf{i})^2} \oint_{\gamma_0} d\zeta_0 \oint_{\gamma_{1}}d\zeta_1\, \big(\frac{1-\zeta^2_0}{1-\zeta^2_1}\big)^n
    \frac{(1+\zeta_0)^{s}}{(1+\zeta_1)^{t}} \frac{\zeta_1^y}{\zeta_0^{x+1}} \frac{1}{\zeta_0-\zeta_1}.
\end{align*}

Let $\chi_{[0,a)}(t_i,x) = \ind{0 \leq x < a_i}$. We shall first show that
\begin{equation} \label{eqn:Kbreakdown}
    \det(I -\chi_a K_n\chi_a)_{\ell^2(\Z \times \{t_1, \ldots, t_k\})} = \det(I -\chi_{[0,a)} K_n\chi_{[0,a)})_{\ell^2(\Z \times \{t_1, \ldots, t_k\})}.
\end{equation}
Note that the r.h.s.~is the determinant of a finite matrix.

Write $\bar{\chi}_0(x,t_i) = \ind{x \geq 0}$. Decompose $K_n = H + J$ where $H(x,s;y,t) = - \ind{s>t,x \geq y}H_{s-t}(x,y)$ (note that $H=0$ unless $x \geq y$).
Then $J(x,s;y,t) = 0$ whenever $x < 0$ because then we can contract the $\zeta_0$-contour to a point. So $J = \bar{\chi}_0 J$.
Now decompose $\chi_a = \chi_1 + \chi_2$ where $\chi_1 = \chi_a \chi_0$ and $\chi_2 = \chi_a \bar{\chi}_0$. Observe that
$\chi_1 H \chi_2(x,y) = H(x,y) \ind{x\geq y, x < 0, y \geq 0} = 0$. Therefore, the kernel $\chi_a K_n \chi_a$ admits a block decomposition
according to the partition $I = \chi_0 + \bar{\chi}_0$ of the space $\ell^2(\Z \times \{t_1,\cdots, t_k\})$:
$$ \chi_a K_n \chi_a = \begin{bmatrix}
    H & 0 \\
    H+J & H+J
\end{bmatrix}.$$
On the $(1,1)$-block above, the kernel $H$ is block triangular with 0s on the diagonal blocks since $H(x,s;y,t)=0$ unless $s > t$. Therefore,
$$\det(I - \chi_a K_n\chi_a) =\det(I - \chi_2 (H+J)\chi_2) = \det(I - \chi_2 K_n \chi_2),$$
which proves \eqref{eqn:Kbreakdown}.

Next, we make explicit choices of contours. We choose the contour $\gamma_0$ as in \eqref{eqn:gamma0} and the contour $\gamma_1$ as in \eqref{eqn:gamma2}.
Along these contours, we change variables $\zeta_i \mapsto z_i/\sqrt{2n}$. Then with $\gamma = \sqrt{2n} \, \gamma_1$,
we find that
$$ - H_{s-t}(x,y)\, \ind{s>t} + \mathcal{J}_n(x,s;y,t) = (\sqrt{2n})^{x-y} \times \mathcal{K}_n(x,s;y,t)$$
where
\begin{align*}
    \mathcal{K}_n(x,s;y,t) &= - \ind{s> t, x \geq y} \frac{((s-t)/\sqrt{2n})^{x-y}}{(x-y)!} \times (1 + O_{x,y}((s-t)^{-1}))\, + \\
    & \frac{1}{(2\pi \mathbf{i})^2} \oint_{\gamma} dz_1 \oint_{|z_0|=1}dz_0\, 
    \left ( \frac{1-\frac{z_0^2}{2n}}{1-\frac{z_1^2}{2n}}\right)^n \frac{(1+\frac{z_0}{\sqrt{2n}})^s}{(1+\frac{z_1}{\sqrt{2n}})^t}\frac{z_1^{y}}{z_0^{x+1}}\, \frac{1}{z_0-z_1}.
\end{align*}
We can remove the conjugation factor $(\sqrt{2n})^{x-y}$ from the kernel.

Now we set $m_i = \lfloor \sqrt{2n} \, t_i \rfloor$ where $t_i$ are as in the statement of Theorem \ref{thm:process}.
We get that
$$\mathbb{P}(\mathcal F_n(t_i) \geq a_i, 1 \leq i \ \leq k) = \det(I - \chi_a \mathcal{K}'_n\chi_a)_{\ell^2(\Z \times \{t_1,\ldots,t_k\})}$$
where $\mathcal{K}'_n(x,s;y,t) = \mathcal{K}_n(x,\lfloor s\sqrt{2n} \rfloor ; y, \lfloor t\sqrt{2n} \rfloor)$.

Arguing as in the proof of Theorem \ref{thm:main}, one sees that
\begin{align*}
    &\frac{1}{(2\pi \mathbf{i})^2} \oint_{\gamma} dz_1 \oint_{|z_0|=1}dz_0\, 
    \left ( \frac{1-\frac{z_0^2}{2n}}{1-\frac{z_1^2}{2n}}\right)^n \frac{(1+\frac{z_0}{\sqrt{2n}})^{\lfloor \sqrt{2n}s\rfloor}}{(1+\frac{z_1}{\sqrt{2n}})^{\lfloor\sqrt{2n}t\rfloor}}\,
    \frac{z_1^{y}}{z_0^{x+1}}\, \frac{1}{z_0-z_1} \to \\
    & \frac{1}{(2\pi \mathbf{i})^2} \oint_{d+\mathbf{i}\R} dz_1 \oint_{|z_0|=1}dz_0\, \frac{e^{\frac{1}{2}z_1^2-tz_1}}{e^{\frac{1}{2}z_0^2-sz_0}}
    \frac{z_1^{y}}{z_0^{x+1}}\, \frac{1}{z_1-z_0}.
\end{align*}
We omit the details here for brevity. This completes the proof of Theorem \ref{thm:process}.

We now turn to the proof of Corollary \ref{cor:extendedkernel}.

Suppose $s > t$. In the definition of $\mathcal{K}(x,s;y,t)$, in the double contour integral, we shift the $z_1$-contour from the vertical line $d+\mathbf{i}\R$ to $-d + \mathbf{i}\R$. We do not encounter any poles at $z_1=0$ because \eqref{eqn:Kbreakdown} ensures $y \geq 0$.
In doing so, we obtain a residue at $z_0 = z_1$, which provides a total contribution of
$$\frac{1}{(2\pi\mathbf{i})} \oint_{|z|=1} dz\, e^{(s-t)z}z^{y-x-1} = \frac{(s-t)^{x-y}}{(x-y)!}\ind{x \geq y}.$$
So it cancels the other term, and we conclude that
\begin{equation} \label{eqn:Kext}
\mathcal{K}(x,s;y,t) = \begin{cases}
    \frac{1}{(2\pi \mathbf{i})^2} \oint_{d+\mathbf{i}\R} dz_1 \oint_{|z_0|=1}dz_0\, \frac{e^{\frac{1}{2}z_1^2-tz_1}}{e^{\frac{1}{2}z_0^2-sz_0}}
    \frac{z_1^{y}}{z_0^{x+1}}\, \frac{1}{z_1-z_0} & s \leq t\\[2ex]
    \frac{1}{(2\pi \mathbf{i})^2} \oint_{-d+\mathbf{i}\R} dz_1 \oint_{|z_0|=1}dz_0\, \frac{e^{\frac{1}{2}z_1^2-tz_1}}{e^{\frac{1}{2}z_0^2-sz_0}}
    \frac{z_1^{y}}{z_0^{x+1}}\, \frac{1}{z_1-z_0} & s > t
\end{cases}.
\end{equation}

Finally, we can represent $\mathcal K$ in terms of Hermite polynomials using \eqref{eqn:Kext}.
Suppose $x,y \geq 0$ and $s \leq t$. By writing $(z_1-z_0)^{-1} = \int_0^{\infty} e^{-\lambda(z_1-z_0)}$
and arguing as in the proof of Corollary \ref{cor:main}, we readily see that
$$\mathcal{K}(x,s;y,t) = \int_0^{\infty}d\lambda\, \frac{H_{x}(s+\lambda)H_{y}(t+\lambda)}{x!}\, \frac{e^{-\frac{(t+\lambda)^2}{2}}}{\sqrt{2\pi}}.$$
When $s > t$, we write $(z_0-z_1)^{-1} = \int_{-\infty}^0 e^{\lambda(z_0-z_1)}$, noting that in this case $\Re(z_0-z_1) > 0$.
Arguing as before using \eqref{eqn:Kext}, we arrive at
$$\mathcal{K}(x,s;y,t) = - \int_{-\infty}^0 d\lambda\, \frac{H_{x}(s+\lambda) H_{y}(t+\lambda)}{x!}\, \frac{e^{-\frac{(t+\lambda)^2}{2}}}{\sqrt{2\pi}}.$$
Upon transposing the kernel, we get the one given in the form of Corollary \ref{cor:extendedkernel}.

\section{Proof of Theorem \ref{thm:extendedcorollary}} \label{sec:SJGUErelation}
Firstly, we note that it is enough to establish the theorem for times $t_1 \leq t_2 \leq \cdots \leq t_k$ and $a_1=1,a_2=2,\ldots,a_k=k$.
Indeed, we may assume w.l.o.g.~that $t_1 \leq t_2 \leq \cdots \leq t_k$. We note that $\mathcal F(t)$ is non-decreasing and, also,
$\lambda_{max}(\mathbf G_{a \times a})$ is non-decreasing in $a$ by Cauchy's interlacing theorem for eigenvalues of Hermitian matrices and their minors.
Therefore, if $a_{i+1} \leq a_i$, then the $(i+1)$-th constraints ($\mathcal{F}(t_{i+1}) \geq a_{i+1}$ and $\mathbf{G}_{a_{i+1}\times a_{i+1}} \leq t_{i+1}$) are redundant.
So after dropping redundant constraints we may assume that $a_1 < a_2 < \cdots < a_k$.
Let $K = a_k$ and observe that, due to $\lambda_{max}(\mathbf G_{a \times a})$ being non-decreasing in $a$,
\begin{align*}
    & \mathbb{P}(\lambda_{max}(\mathbf G_{a_i \times a_i}) \leq t_i\, \forall i) =\\
& \mathbb{P}(\lambda_{max}(\mathbf G_{1 \times 1})\leq t_1, \ldots, \lambda_{max}(\mathbf G_{a_1 \times a_1}) \leq t_1, \lambda_{max}(\mathbf G_{a_1+1 \times a_1+1}) \leq t_2, \ldots \lambda_{max}(\mathbf G_{a_2 \times a_2})\leq t_2, \\
& \quad \ldots, \lambda_{max}(\mathbf G_{a_{k-1}+1 \times a_{k-1}+1}) \leq t_k, \ldots, \lambda_{max}(\mathbf G_{K \times K})\leq t_k).
\end{align*}
Thus, we may assume that $a_i=i$ and $t_1 \leq t_2 \leq \cdots \leq t_k$.

We note that the formula from Theorem \ref{thm:process} continues to hold if we assume that the $t_i$s are weakly ordered; we assumed strict ordering since one can always reduce to that case. Alternatively, we can prove the theorem for strictly ordered $t_i$s and $a_i=i$, and then take limits in the $t_i$s to get equality for weakly ordered $t_i$s.

Write
$$w(x)=\frac{e^{-x^2/2}}{\sqrt{2\pi}}.$$
The Hermite polynomials then satisfy
$$H_n(x)w(x)=(-1)^n w^{(n)}(x)$$
where $w^{(n)}$ is the $n$-fold derivative of $w$.

For $n\in\mathbb Z$, define
\begin{equation} \label{eqn:Fn}
w_n(t)=
\begin{cases}
\displaystyle
\int_{-\infty}^{t} dx\, \frac{(t-x)^{n-1}}{(n-1)!}\,w(x), & n\ge1,\\[2ex]
\displaystyle
w^{(-n)}(t), & n\le0.
\end{cases}
\end{equation}
Then $w_n'(t)=w_{n-1}(t)$ for every $n\in\mathbb Z$.

\begin{proposition} \label{prop:SJGUE}
With the notation above,
$$
\det(I-\chi\Jext\chi) = \det\bigl[w_{j-i+1}(t_i)\bigr]_{i,j=1}^k.
$$
\end{proposition}

\begin{proof}
Set
$$ \psi_p(x)=H_p(x)w(x), \qquad \varphi_q(x)=\frac{H_q(x)}{q!}.$$
The Hermite orthogonality relation reads
$$\int_{\R}\psi_p(x)\varphi_q(x)\,dx= \ind{p=q}.$$

Observe that $I-\chi\Jext\chi$ is a finite matrix. The matrix is indexed by pairs
$$(i,p),\qquad 1\le i\le k,\quad 0\le p<i.$$
Thus, it is a $k \times k$ block matrix where block $(i,j)$ is indexed by $[(i,p), (j,q)]$ for $0 \leq p < i$ and $0 \leq q < j$.
The dimension of the matrix is $1+2+\cdots+k = \binom{k+1}{2}$.

We rewrite its entries in a uniform manner. Define
$$A_{ij}(p,q) = \int_{-\infty}^{0} \psi_p(t_i+u)\varphi_q(t_j+u)\,du,$$
and
$$B_{ij}(p,q) = \int_{\R} \psi_p(t_i+u)\varphi_q(t_j+u)\,du.$$
Using the two cases in the definition of $\Jext$, together with Hermite orthogonality when $i=j$, one gets
\begin{equation}
\label{eq:matrix-entry}
M_{(i,p),(j,q)} = A_{ij}(p,q)-\1_{\{i>j\}}B_{ij}(p,q),
\end{equation}
where $M$ denotes the finite matrix of $I-\chi\Jext\chi$.

We now simplify the columns. For each fixed $j$, the polynomials
$$\varphi_0(t_j+u),\ldots,\varphi_{j-1}(t_j+u)$$
have degrees $0,\ldots,j-1$, with leading coefficients
$$\frac1{0!},\ \frac1{1!},\ \frac1{2!},\ldots,\frac1{(j-1)!}.$$
Therefore, the change of basis
$$\varphi_q(t_j+u) \longleftrightarrow \frac{u^q}{q!},\qquad 0\le q<j,
$$
is triangular with diagonal entries equal to $1$. Performing this change of basis in each column block does not change the determinant.

After this change of basis, define
$$ A_i(p,q) = \int_{-\infty}^{0} \psi_p(t_i+u)\frac{u^q}{q!}\,du, $$
and
$$ B_i(p,q) = \int_{\R} \psi_p(t_i+u)\frac{u^q}{q!}\,du.$$
The new matrix, which is denoted $\widehat M$, has entries
\begin{equation}
\label{eq:Mhat}
\widehat M_{(i,p),(j,q)} = A_i(p,q)-\1_{\{i>j\}}B_i(p,q).
\end{equation}
The important point is that $A_i(p,q)$ and $B_i(p,q)$ no longer depend on the column index $j$.

We next record a simple fact about $B_i(p,q)$. Since
$$  \psi_p(t_i+u)=(-1)^p\frac{d^p}{du^p}w(t_i+u),$$
integration by parts gives
$$ B_i(p,q) = \int_{\R} du\, w(t_i+u) \frac{d^p}{du^p}\Bigl(\frac{u^q}{q!}\Bigr).$$
Hence
\begin{equation}
\label{eq:Btriangular}
B_i(p,q)=0\quad\text{if }p>q, \qquad B_i(p,p)=1.
\end{equation}
Thus the matrix $[B_i(p,q)]_{0\le p,q\le m}$ is upper triangular with diagonal entries equal to $1$.

We now eliminate repeated columns. Let $C_{j,q}$ denote the column of $\widehat M$ indexed by $(j,q)$. Starting with $j=k$ and proceeding downwards to $j=2$, perform the column operations
$$ C_{j,q}\longleftarrow C_{j,q}-C_{j-1,q}, \qquad 0\le q\le j-2.
$$
Since we proceed in decreasing order of $j$, the column $C_{j-1,q}$ has not yet been modified.

From \eqref{eq:Mhat}, the new column has entries
\begin{align*}
&\widehat M_{(i,p),(j,q)} - \widehat M_{(i,p),(j-1,q)} \\
&\qquad = \bigl(\1_{\{i>j-1\}}-\1_{\{i>j\}}\bigr)B_i(p,q) \\
&\qquad = \1_{\{i=j\}}B_j(p,q).
\end{align*}

Thus this column is supported only on the rows belonging to the $j$th block. Moreover, by \eqref{eq:Btriangular}, the columns with $0\le q\le j-2$ involve only the rows $(j,p)$ with $0\le p\le j-2$, and on those rows they form an upper triangular matrix with diagonal entries equal to $1$.
Consequently, these rows and columns contribute a determinant equal to $1$ and may be removed without changing the determinant of $\widehat M$. In more detail, after permuting columns and the corresponding rows of $\widehat M$ (which does not change the determinant), it takes the form
$$ \widehat M = \begin{bmatrix} D' & 0 \\ * & D\end{bmatrix}, \qquad \det[D']=1.$$
So $\det[\widehat M] = \det[D]$.

After this elimination has been carried out, the only surviving column in the $j$th block is $(j,j-1)$.
Hence the original determinant has been reduced to a $k\times k$ determinant of a matrix $D$ as above.

The surviving columns were never altered. Therefore, the reduced matrix has entries
$$
D_{ij} =
A_i(i-1,j-1) - \1_{\{i>j\}}B_i(i-1,j-1).$$
If $i>j$, \eqref{eq:Btriangular} gives $B_i(i-1,j-1)=0$.
Thus,
\begin{equation}
\label{eq:Dmoment}
D_{ij} = \int_{-\infty}^{0} du\, H_{i-1}(t_i+u)w(t_i+u) \frac{u^{j-1}}{(j-1)!}.
\end{equation}
We have therefore proved that
$$ \det(I-\chi\Jext\chi)=\det[D_{ij}]_{i,j=1}^k. $$

It remains to identify $D_{ij}$ with $w_{j-i+1}(t_i)$ up to a conjugation by signs.
For $j\ge1$, changing variables $x=t+u$ in the definition of $w_j$ gives
$$ w_j(t) = \int_{-\infty}^{0} du\, \frac{(-u)^{j-1}}{(j-1)!} w(t+u).$$
Differentiate this identity $i-1$ times with respect to $t$. Since
$$ H_{i-1}(x)w(x)=(-1)^{i-1}w^{(i-1)}(x),$$
we obtain
\begin{align*}
D_{ij} &= (-1)^{i+j-2}w_j^{(i-1)}(t_i) \\
&= (-1)^{i+j-2} w_{j-i+1}(t_i) \\
&= (-1)^{j-i}w_{j-i+1}(t_i).
\end{align*}
Thus,
$$\det[D_{i,j}]_{i,j=1}^k = \det\bigl[w_{j-i+1}(t_i)\bigr]_{i,j=1}^k,$$
as claimed.
\end{proof}

In order to complete the proof, we need to introduce the model of Brownian last passage percolation.
Let $B_1(s), B_2(s), \ldots$ be independent, standard Brownian motions. Brownian last passage percolation
is a sequence of processes $X_1(s) \leq X_2(s) \leq X_3(s) \leq \cdots$ for $s \geq 0$. It is defined according to
\begin{equation} \label{eqn:BLPP}
    X_k(s) = \sup_{0=s_0 \leq s_1 \leq s_2 \leq \cdots \leq s_k = s}\, \sum_{i=1}^k B_i(s_i) - B_i(s_{i-1}).
\end{equation}
Geometrically, $X_1 = B_1$, $X_2$ is the reflection (in the Skorokhod sense) of $B_2$ off $X_1$, $X_3$ is the reflection of $B_3$ off $X_2$, and so on.

Warren \cite[Proposition 8]{Warren} proves that $(X_1,\ldots, X_k)$ is a Markov process with an explicit transition kernel.
In particular, at time $s=1$, for every $x_1 \leq x_2 \leq \cdots \leq x_k \in \R$,
\begin{equation} \label{eqn:warren}
\mathbb{P}(X_1(1) \in dx_1, X_2(1) \in dx_2, \ldots, X_k(1) \in dx_k) = \det[w_{i-j}(x_j)]_{i,j=1}^k\, dx_1 \cdots dx_k.
\end{equation}
Recall $w_n$ from \eqref{eqn:Fn}; it satisfies $\int_{-\infty}^tdx\, w_n(x) = w_{n+1}(t)$ for every $n \in \Z$.
Suppose $t_1 \leq t_2 \leq \cdots \leq t_k$. By integrating \eqref{eqn:warren} over the appropriate domain and using basic properties of the determinant, we find that
\begin{equation} \label{eqn:BLPPcdf}
   \mathbb{P}(X_1(1) \leq t_1, X_2(1) \leq t_2, \ldots, X_k(1) \leq t_k) = \det[w_{i-j+1}(t_j)]_{i,j=1}^k.
\end{equation}
Note that
$$\det[w_{i-j+1}(t_j)]_{i,j=1}^k = \det[w_{j-i+1}(t_i)]_{i,j=1}^k.$$
Proposition \ref{prop:SJGUE} together with the above and Theorem \ref{thm:process} (together with Corollary \ref{cor:extendedkernel}) imply
\begin{equation} \label{eqn:preGUE}
    \mathbb{P}(\mathcal{F}(t_1) \geq 1, \ldots, \mathcal{F}(t_k) \geq k) = \mathbb{P}(X_1(1) \leq t_1, \ldots, X_k(1) \leq t_k).
\end{equation}
Baryshnikov \cite[Theorem 0.7]{Bar} proves that the vector $(X_1(1), X_2(1), \ldots)$ has the same law as the vector
$(\lambda_{max}(\mathbf G_{1 \times 1}), \lambda_{max}(\mathbf G_{2 \times 2}), \ldots)$. Therefore, \eqref{eqn:preGUE} establishes Theorem \ref{thm:extendedcorollary}.

\section*{Acknowledgements}
I thank Julian Ransford for a discussion that led me to consider the subject of this article.
Thanks also to B\'alint Vet\H{o} for clarifying the relationship between the Sepp\"al\"ainen-Johansson model
and the Brownian web distance, and sending me an advance copy of \cite{RVV}.

\end{document}